\documentclass{amsproc}

\usepackage{amssymb}
\usepackage{graphicx}
\usepackage{tikz}
\usepackage[cmtip,all]{xy}
\usetikzlibrary{fit,matrix}

\usepackage{color}
\usepackage{xcolor}
\usepackage{cleveref}
\usepackage{multicol}
\usepackage{tabularx}
\usepackage{booktabs}
\usepackage{mathtools}
\usepackage{url}
\usepackage{bm}
\usepackage{ulem}
\usepackage{framed}

\newtheorem{theorem}{Theorem}[section]
\newtheorem{lemma}[theorem]{Lemma}
\newtheorem{proposition}[theorem]{Proposition}

\theoremstyle{definition}
\newtheorem{definition}[theorem]{Definition}
\newtheorem{example}[theorem]{Example}
\newtheorem{conjecture}[theorem]{Conjecture}

\theoremstyle{remark}
\newtheorem{remark}[theorem]{Remark}

\numberwithin{equation}{section}

\title[Deformations via Hilbert-Burch matrices]{Deformations of local Artin rings via Hilbert-Burch matrices}

\author{Roser Homs}
\address{Centre de Recerca Matem\`atica}
\curraddr{}
\email{rhoms@crm.cat}
\thanks{}

\author{Anna-Lena Winz}
\address{Freie Universit\"at Berlin, Arnimallee 3, 14195 Berlin, Germany}
\curraddr{}
\email{anna-lena.winz@fu-berlin.de}
\thanks{}


\date{}

\newcommand{\m}{\mathfrak m}
                   
\def\res{{\bf k}}

\def\HF{\operatorname{HF}}

\def\Hilb{\operatorname{Hilb}}

\def\ord{{\mathrm{ ord}}}

\def\otau{{\overline{\tau}}}

\newcommand{\Lt}{\operatorname{Lt_{\overline{\tau}}}}

\newcommand{\T}{\mathcal{M}}
\newcommand{\Nd}{\mathcal{N}_{<\underline{d}}}

\newcommand{\Ltlex}{\operatorname{Lt_{lex}}}
\newcommand{\Ltau}{\operatorname{Lt_{\tau}}}

\newcommand{\lex}{\operatorname{lex}}
\newcommand{\Lex}{\operatorname{Lex}}
\newcommand{\Gin}{\operatorname{Gin}}

\newcommand{\rk}{\operatorname{rk}}

\begin{document}

\begin{abstract} 
In the local setting, Gr\"obner cells are affine spaces that parametrize ideals in $\res[\![x,y]\!]$ that share the
same leading term ideal with respect to a local term ordering. In particular, all ideals in a cell have the same Hilbert function, so they provide a cellular decomposition of the punctual Hilbert scheme compatible with its Hilbert function stratification. We exploit the parametrization given in \cite{HW21} via Hilbert-Burch matrices to compute the Betti strata, with hands-on examples of deformations that preserve the Hilbert function, and revisit some classical results along the way. Moreover, we move towards an explicit parametrization of all local Gr\"obner cells.
\end{abstract}

\maketitle

\section{Introduction}

The punctual Hilbert scheme $\Hilb^n(\res[\![x,y]\!])$ is a fundamental object in algebraic geometry that parametrizes ideals of colength $n$ in $\res[\![x,y]\!]$. During the decades of the 70s and 80s of the past century several authors \cite{Bri77, BG74,ES,Gal74,Gra83,Iar77,Ya89} studied its geometrical properties. In particular, Briançon \cite{Bri77} and Iarrobino \cite{Iar77} proved that $\Hilb^n(\res[\![x,y]\!])$ is irreducible of dimension $n-1$ for $\res$ algebraically closed.

A natural approach to understand these schemes better is decomposing them into smaller spaces by considering sets of flat deformations of a given colength \(n\) monomial ideal. On one hand, Briançon and Galligo \cite{BG74} introduced the technique of vertical
standard basis stratification of \(\Hilb^n\res[[x, y]]\) consisting on the study of the generators of such deformations. Yaméogo \cite{Ya89} showed that each vertical stratum of \(\Hilb^n\res[[x, y]]\) is isomorphic to an affine space of dimension given by a very simple formula.

On the other hand, Conca and Valla \cite{CV08} focused on studying the syzygies of the deformations of monomial ideals. For a monomial ideal $E$, they used Hilbert-Burch matrices to parametrize the so-called Gr\"obner cell \(V_{\lex}(E)\) of all ideals $J$ with leading term ideal \(\Ltlex(J)=E\) with respect to the lexicographical term ordering $\lex$. When the Hilbert scheme is smooth these cells are affine spaces by \cite{Bia73}. A term ordering \(\tau\) determines a 
\(\res^*\)-action on the punctual Hilbert scheme with isolated fixed points --- the monomial
ideals --- and the \(V_{\tau}(E)\) are all those points in the Hilbert scheme specialising to 
the point \(E\). When one has a smooth scheme with a \(\res^*\)-action with isolated fixed 
points, one obtains a cellular decomposition of that space. This approach has been used to study certain invariants --- such as Betti numbers --- of Hilbert schemes of points of \(\mathbb{A}^2\), \(\mathbb{P}^2\) and \((\mathbb{A}^2,0)\), see \cite{ES,Got90}.

The cellular decomposition of \(\Hilb^n\res[[x, y]]\) resulting from
vertical stratification or equivalently Gr\"obner cells is not compatible with the
Hilbert function stratification of \(\Hilb^n\res[[x, y]]\). In other words, colength \(n\) ideals with different Hilbert functions may belong to the same Gr\"obner cell. Bulding on the works of Rossi and Sharifan \cite{RS10} and Bertella \cite{Ber09} on producing zero-dimensional ideals in \(\res[[x, y]]\) with a given Hilbert function, the authors of the present paper consider in \cite{HW21} a local term ordering $\otau$ to define Gr\"obner cells that overcome this problem.The parametrization of $V_\otau(E)$ in terms of Hilbert-Burch matrices has been completely described for a special class of monomial ideals, that includes the well-studied lex-segment ideals, see \Cref{thm} or \cite[Theorem 5.7]{HW21}.

The goal of the present contribution is twofold: (1) to exploit the parametrization of Gr\"obner cells for lex-segment ideals to study the Betti strata and homogeneous subcells of the subscheme $\Hilb^h(\res[\![x,y]\!])$ of $\Hilb^n(\res[\![x,y]\!])$ that parametrizes ideals with a given Hilbert function \(h\); and (2) to move towards a complete description of a cellular decomposition of $\Hilb^n(\res[\![x,y]\!])$ that is compatible with the Hilbert function stratification.

The paper is structured as follows. \Cref{sec:cells} introduces the machinery needed for studying Gr\"obner cells in the local setting and recalls the fundamental result of \cite{HW21}, namely the parametrization of Gr\"obner cells of lex-segment ideals in terms of canonical Hilbert-Burch matrices.
In \Cref{sec:parametrization}, we focus on the use --- in both practise and theory --- of such matrices to obtain deformations of the lex-segment ideal with desired properties, such as a given minimal number of generators or homogeneity. The main results of this section are \Cref{th:sets_d_gens} and \Cref{prop:hom}, where we parameterize the Betti strata and the homogeneous sub-cell, respectively, of the Gr\"obner cell for lex-segment ideals. Our journey towards these statements includes revisiting classical results from the perspective of canonical Hilbert-Burch matrices.
Finally, \Cref{sec:conjecture} is devoted to cellular decompositions of the punctual Hilbert scheme $\Hilb^n(\res[\![x,y]\!])$ that are compatible with its Hilbert function stratification. We explicitly compute such a decomposition for $n=6$ and briefly describe the Julia programming language module we created to compute Gr\"obner cells of any monomial ideal according to the parametrization in \cite[Conjecture 5.14]{HW21}. We provide computational evidence for the correctness of the conjecture.

\section{Gr\"obner cells in the local setting}\label{sec:cells}

In this section we will introduce the notation and recall some basics on local term orderings. 
We will finish with our main result from \cite{HW21}, which will be applied in the following sections 
to prove and reprove some properties of the punctual Hilbert scheme and its points, codimension two ideals in the power series ring.

\begin{definition}
A term ordering $\tau$ in the polynomial ring $P=\res[x_1,\dots,x_n]$ induces a reverse-degree ordering $\otau$ in $R=\res[\![x_1,\dots,x_n]\!]$ such that for any monomials $m,m'$ in $R$, $m>_{\otau} m'$ if and only if
$$\deg(m) < \deg(m') \mbox{ or }\deg(m) = \deg(m') \mbox{ and } m >_{\tau} m'.$$
\noindent
We call $\otau$ the \textbf{local term ordering} induced by the global ordering $\tau$. 
\end{definition} 

\begin{definition}\label{def:stdbasis} Given an ideal $J$ of $R$, we define the \textbf{leading term ideal} of $J$ as the monomial ideal in $P$ generated by the leading terms with respect to the local degree ordering $\otau$, i.e. $\Lt(J)=\left( \Lt(f): f\in J\right)\subset\res[x_1,\dots,x_n]$.

We call a subset $\lbrace f_1,\dots,f_m\rbrace$ of $J$ a \textbf{$\otau$-enhanced standard basis of $J$} if $\Lt(J)=(\Lt(f_1),\dots,\Lt(f_m))$.
\end{definition}

From the computational point of view, the tangent cone algorithm given by Mora in \cite{Mor82} --- a variant of Buchberger's algorithm to compute Gr\"obner bases of ideals in the polynomial ring --- allows us to explicitly obtain $\otau$-enhanced standard basis of ideals in $\res[\![x_1,\dots,x_n]\!]$ that are generated by polynomials. See \cite[Sections 1.6,1.7,6.4]{GP08} for a detailed treatment and implementation.

From now on we will restrict to the codimension two case. Let \(P = \res[x,y]\), \(R = \res[\![x,y]\!]\) and fix \(x>_\tau y\).

\begin{definition} \label{def:GroebnerCell}
Given a zero-dimensional monomial ideal $E$ in $R$, we denote by $V(E):=\lbrace J\subset R: \Lt(J)=E\rbrace$ the \textbf{Gr\"obner cell} of $E$ with respect to the local term ordering $\otau$. 
\end{definition}

As the name suggests  these are affine spaces by results about group actions of \(\res^*\) on smooth varieties with isolated fixed points of \cite[Theorem~4.4]{Bia73}. 
They have been explicitly described --- at least for a special class of monomial ideals --- for the lexicographical term ordering \cite{CV08}, the degree lexicographical term ordering \cite{Con11} and the local term ordering induced by both of them \cite{HW21}. 

\begin{remark} \label{rem:localstructure}
It is important to note that, as opposed to what happens for global term orderings, ideals in a Gr\"obner cell with respect to a local term ordering preserve the Hilbert function of the monomial ideal, see \cite[Corollary~5.10]{HW21}. 
Indeed, by definition of local term ordering $\Lt(f)=\Ltau(f^\ast)$, where $f^\ast$ is the initial form of $f$. In particular, any $\otau$-enhanced standard basis $\lbrace f_1,\dots,f_m\rbrace$ of $J$ is also a standard basis of $J$, namely the initial ideal $J^\ast$ is generated by the initial forms $f_1^\ast,\dots,f_m^\ast$. Therefore, $\Ltau(J^\ast)=\Lt(J)$ and hence
$\HF_{R/J}=\HF_{P/J^\ast}=\HF_{P/ \Lt(J)}$.
So these Gr\"obner cells are compatible with the Hilbert function stratification of \(\Hilb^n(\res[\![x,y]\!])\) and induce a cellular decomposition of the strata.
\end{remark}

The strategy used in the description of the Gr\"obner cells of the three previously mentioned term orderings consists in focusing directly on the syzygies instead of the generators themselves and to that aim it is convenient to define the notion of canonical Hilbert-Burch matrix.

Note that any zero-dimensional monomial ideal in $\res[\![x,y]\!]$ can be written as $E=(x^t,\dots,x^{t-i}y^{m_i},\dots,y^{m_t})$ with $0=m_0< m_1\leq \dots\leq m_t$.

\begin{definition}\label{canonical}
The \textbf{canonical Hilbert-Burch matrix}\index{canonical Hilbert-Burch matrix} of the monomial ideal $E=(x^t,\dots,x^{t-i}y^{m_i},\dots,y^{m_t})$ is the Hilbert-Burch matrix of $E$ of the form

$$H=\left(\begin{array}{cccc}
y^{d_1} & 0 & \cdots & 0\\
-x & y^{d_2} & \cdots & 0\\
0 & -x & \cdots & 0\\
\vdots & \vdots & & \vdots\\
 0 & 0 & \cdots & y^{d_t}\\
0 & 0 & \cdots & -x
\end{array}\right),$$
\noindent
where $d_i=m_i-m_{i-1}$ for any $1\leq i\leq t$. 
The \textbf{degree matrix}\index{degree matrix} $U$ of $E$ is the $(t+1)\times t$ matrix with integer entries $u_{i,j}=m_j-m_{i-1}+i-j$, for $1\leq i\leq t+1$ and $1\leq j\leq t$.
\end{definition}

It follows from the definition that $u_{i,i}=d_i$ and $u_{i+1,i}=1$, for $1\leq i\leq t$.
 
\begin{definition}\label{def:setT2}
We denote by $\T(E)$ the set of matrices $N=(n_{i,j})$ of size $(t+1)\times t$ with entries in $\res[y]$ such that
\begin{itemize}
\item $n_{i,j}=0$ for any $i\leq j$, 
\item and for \(i>j\) the non-zero entries satisfy $u_{i,j} \leq \ord(n_{i,j}) \leq\deg(n_{i,j})<d_j$.
\end{itemize}
\end{definition}

For the time being we will focus on a special class of monomial ideals called lex-segment ideals. 
\begin{definition}\label{def:lex}
 A monomial ideal $L=(x^t,\dots,x^{t-i}y^{m_i},\dots,y^{m_t})$ is a \textbf{lex-segment ideal} if and only if $0<m_1<\dots<m_t$.
\end{definition}

By Macaulay's theorem \cite{Mac27}, any Hilbert function $h$ has a unique lex-segment ideal $\Lex(h)$, see \cite[Theorem 6.3.1]{HH} for a modern treatment. Consider an ideal $J$ with Hilbert function $h$. In characteristic 0, after possibly a generic change of coordinates on $J$, $\Lt(J)$ coincides with $\Lex(h)$. In other words, the lex-segment ideal is the generic initial ideal with respect to $\otau$. For the reader familiar with generic initial ideals in the global setting, $\Gin_{\otau}(J)=\Gin_{\tau}(J^\ast)$. See \cite[Section 1.4]{Ber09} for an exposition about the notion of generic initial ideal in the local setting.

\begin{theorem}\label{thm}\cite[Theorem~5.7]{HW21} Given a zero-dimensional lex-segment monomial ideal $L\subset R$ with canonical Hilbert-Burch matrix $H$, the map
$$\begin{array}{rrcl}
\Phi_L: & \T(L) & \longrightarrow & V(L)\\
& N & \longmapsto & I_t(H+N)
\end{array}$$

\noindent 
is a bijection, where $I_t(M)$ denotes the ideal of $t$-minors of the matrix $M$. 

\end{theorem}

This theorem allows us to define the canonical Hilbert-Burch matrix of any zero-dimensional ideal $J\subset R$ that can be obtained as a deformation of a lex-segment ideal $L$ as $H+\Phi_L^{-1}(J)$, where $H$ is the canonical Hilbert-Burch matrix of the monomial ideal $\Lt(J)$. In particular, the maximal minors of $H+\Phi_L^{-1}(J)$ form a \(\otau\)-enhanced standard basis of \(J\).

In characteristic 0, the affine space $V(L)$ encodes all ideals with a given Hilbert function up to generic change of coordinates. In other words, the map $\Phi_L$ provides a parametrization via canonical Hilbert-Burch matrices of all ideals with Hilbert function $h=\HF_{P/L}$ up to generic change of coordinates.
For a formal version of this statement, see \cite[Corollary 5.10]{HW21}. See \Cref{sec:conjecture} for a discussion about the description of Gr\"obner cells and the notion of canonical Hilbert-Burch matrix beyond the lex-segment case.

\section{Parametrizing ideals with a given Hilbert function}\label{sec:parametrization}

In this section we will showcase how to exploit canonical Hilbert-Burch matrices to obtain all ideals with a given Hilbert function, and special additional properties, up to generic change of coordinates. 
In the first part of the section, we will exhibit hands-on examples on how to refine the parametrization from \Cref{thm} by adding constraints on the ideals such as the minimal number of generators or homogeneity. 
The second part is devoted to formalizing these procedures. We will revisit some classical results on admissible numbers of generators and on certain subschemes of the punctual Hilbert scheme $\Hilb^n(\res[\![x,y]\!])$, and give alternative proofs based on our canonical Hilbert-Burch matrices. This will assist us in reaching \Cref{th:sets_d_gens} and \Cref{prop:hom}, where we provide explicit descriptions of the affine and quasi-affine varieties forming the Betti strata and homogeneous cells, respectively.

\medskip

In what follows $h$ is an admissible Hilbert function for some zero-dimensional ideal in $R=\res[\![x,y]\!]$, namely $h=(1,2,\dots,t,h_t,\dots,h_s,0)$, with $t\geq h_t\geq \dots\geq h_s$, see \cite{Mac27}. 
Moreover, we will denote by $\Delta(h)$ the maximal jump in the Hilbert function $h$, that is $\Delta(h):=\max \lbrace \vert h_i-h_{i-1}\vert: 1\leq i\leq s\rbrace$. For an ideal \(J\) 
we denote by \(\mu(J)\) the minimal number of generators of \(J\). We will say that a minimal number of generators \(d\) is admissible for a Hilbert function \(h\) if there exists an ideal \(J\) with Hilbert function \(h\) and \(\mu(J)=d\).

Let $L=(x^t,x^{t-1}y^{m_1},\dots,y^{m_t})$ be the associated lex-segment ideal to the Hilbert function \(h\) and let us focus on special subsets of $V(L)$: 

\begin{definition} \label{def:BettiStrata}
For any admissible minimal number of generators $d$ of the Hilbert function $h$, we call $V_d(L)=\lbrace J\in V(L):\mu(J)=d\rbrace$ the $d$-\textbf{Betti stratum} of $V(L)$. 
The set \(V_2(L)\) is the subset of \(V(L)\) consisting of complete intersection ideals. Therefore, we will also refer to it as \(V_{\textrm{CI}}(L)\).

We denote the homogeneous sub-cell of $V(L)$ by $V_{\textrm{hom}}(L)=\lbrace J\in V(L): J=J^\ast\rbrace$.
\end{definition}
In \cite[Remark 4.7]{RS10}, Rossi and Sharifan give a procedure to obtain special instances of  ideals 
in the \(d\)-Betti stratum \(V_d(L)\). The authors considered very specific deformations of a Hilbert-Burch matrix \(H\) of the lex-segment ideal \(L\): whenever $u_{i+2,i}\leq 0$, the entry $n_{i+2,i}$ is set to \(1\) and all other entries are set to \(0\). Note that a matrix \(N\) obtained this way belongs to \(\T(L)\), hence the ideal \(I_t(H+N)\) is in  \(V(L)\).
This approach is a particular case of the fact that the minimal number of generators of $J\in V(L)$ decreases as the rank of the matrix $\bar{N}$ formed by the constant terms of $N\in \T(L)$ increases. We reproduce a tailored version of a result of Bertella \cite[Lemma 2.1]{Ber09} and its proof here for the sake of completeness.

\begin{lemma}\label{lem:rk} Let $L$ be a lex-segment ideal with canonical Hilbert-Burch matrix $H$ and consider $J\in V(L)$. Then the minimal number of generators $\mu(J)=t+1-\rk(\bar{N})$, where $N$ is the unique matrix in $\T(L)$ such that $J=I_t(H+N)$.
\end{lemma}

\begin{proof} Since the columns of the matrix $H+N$ encode a system of generators of the syzygies of its maximal minors (that is, of the generators of the ideal $J$), the sequence $R^t\xrightarrow{H+N} R^{t+1}\longrightarrow J \longrightarrow 0$ is exact. Tensoring with $R/\m$ over $R$ yields 

$$(R/\m)^t\xrightarrow{\bar{N}} (R/\m)^{t+1}\longrightarrow J/\m J \longrightarrow 0,$$

\noindent
hence the minimal number of generators of $J$ is $\dim_{R/\m} J/\m J=t+1-\rk(\bar{N})$.
\end{proof}

In the following example, we will show how the parametrization of the Gr\"obner cell $V(L)$ via canonical Hilbert-Burch matrices generalizes the constructions of Rossi and Sharifan in \cite{RS10}.

 \begin{example} \label{ex:3.3}
 Consider the lex-segment ideal 
\(L = (x^4,x^3y, x^2y^5, xy^8, y^{10})\) with Hilbert function $h=(1,2,3,4,3,3,3,2,2,1)$, canonical Hilbert-Burch matrix $H$ and degree matrix $U$:

\[H=\left(
\begin{array}{c c c c}
y & 0 & 0 & 0\\
-x & y^4 & 0 & 0\\
0 & -x & y^3 & 0 \\
0 & 0 & -x & y^2 \\
0 & 0 & 0 & -x\\
\end{array} \right)
\hspace{1.5cm}
U =
\left(
\begin{array}{c c c c}
1  &   4 &    6 &    7\\
     1  &   4  &   6 &    7\\
    -2  &   1  &   3  &   4\\
    -4  &  -1  &   1   &  2\\
    -5  &  -2  &   0   &  1 \\
\end{array} \right).
\]

By \Cref{thm}, any \(J \in V(L)\)
is the ideal of maximal minors of the matrix 
\[
H+N = 
\left(
\begin{array}{c c c c}
y& 0&          0&      0\\  
-x&y^4&         0&       0\\ 
c_1& -x+c_2 y+c_3 y^2+c_4 y^3&y^3&     0\\ 
c_5& c_6+c_7y+c_8y^2+c_9y^3& -x+c_{10}y+c_{11}y^2&y^2\\
c_{12}& c_{13}+c_{14}y+c_{15}y^2+c_{16}y^3& c_{17}+c_{18}y+c_{19}y^2& -x+c_{20}y\\
\end{array} \right),
\]
for some coefficients \(c_1, \dots, c_{20} \in  \res\). In other words, $V(L)$ is a 20-dimensional affine space with origin the lex-segment ideal $L$. Any ideal of the form $J=I_4(H+N)$ can therefore be identified with the point with coordinates $(c_1,\dots,c_{20})$ in $\mathbb{A}^{20}$. 

Note that the coordinates $c_2,c_{10},c_{17},c_{20}$ play a special role: $J$ is homogeneous if and only if those are the only non-zero coordinates. Indeed, they are the coefficients of degree matching the corresponding entry of the degree matrix. Thus \(V_{\textrm{hom}}\) as defined in \Cref{def:BettiStrata} 
is the subset of \(V(L)\) where all other coordinates vanish. Let \(\mathcal{I} \subset \res[c_1,\dots,c_{20}]\) be an ideal, then \(\mathcal{V}(\mathcal{I}) \subset \mathbb{A}^{20}\) will denote the vanishing set of \(\mathcal{I}\). And we can formalize the previous statement to:

\[V_{\textrm{hom}}(L)=\mathcal{V}(c_1,c_3,c_4,c_5,c_6,c_7,c_8,c_9,c_{11},c_{12},c_{13},c_{14},c_{15},c_{16},c_{18},c_{19})\simeq \mathbb{A}^4.\]
Since a $\otau$-enhanced standard basis is in particular a standard basis, it offers a very simple procedure for computing the initial ideal $J^\ast$ for a non-homogeneous $J$: set all $c_i=0$ except for $i=2,10,17,20$, namely project onto $\mathbb{A}^4$. 

Let us revisit the specific deformation of \(L\) given in \cite[Example 4.8]{RS10}. The authors constructed the complete intersection
ideal \(J=(x^4-x^2y^2-x^2y^3-x^2y^4+y^6, x^3y-xy^3-xy^4) \in V(L)\) and calculated \(J^\ast\). 
Note that $J=I_4(H+N)$ with \(N \in \T(L)\) given by $c_1=c_6=c_{17}=1$ and 0's otherwise. Indeed, $J^\ast$ is obtained by keeping $c_{17}=1$ and setting 0's otherwise. 

Note that the third column of the canonical Hilbert-Burch matrix of $J$ allows us to express $f_4$ as a combination of $f_2,f_3$, the second column yields $f_3$ as a combination of $f_1$ and $f_2$ and the first one gives $f_2$ as a combination of $f_0$ and $f_1$. In fact, any ideal obtained from a matrix satisfying $c_1c_6c_{17}\neq 0$ is a complete intersection generated by $f_0$ and $f_1$. 

By \Cref{lem:rk}, this is actually the only way to obtain complete intersections in $V(L)$.
$V_{\textrm{CI}}(L)$ (as in \Cref{def:BettiStrata}) is a full-dimensional quasi-affine variety obtained by removing the union of the three hyperplanes $\lbrace c_1=0\rbrace$, $\lbrace c_6=0\rbrace$ and $\lbrace c_{17}=0\rbrace$ from $\mathbb{A}^{20}$. 

Similarly, the sets of ideals in $V(L)$ with exactly 5,\,4 and 3 generators can be described as follows:

$$V_5(L)=\mathcal{V}(c_1,c_5,c_6,c_{12},c_{13},c_{17})\simeq \mathbb{A}^{14},$$
$$V_4(L)=\mathcal{V}(c_1c_6,c_1c_{13},c_1c_{17},c_5c_{17},c_6c_{17},c_5c_{13}-c_6c_{12})\backslash V_5(L),$$
$$V_3(L)=\mathcal{V}(c_1c_6c_{17})\backslash \mathcal{V}(c_1c_6,c_1c_{13},c_1c_{17},c_5c_{17},c_6c_{17},c_5c_{13}-c_6c_{12}).$$

In particular, homogeneous ideals with Hilbert function $h$ will always have either 4 or 5 generators, since the only non-zero constant term of $H+N$ allowed in $V_{\textrm{hom}}(L)$ is $c_{17}$. 

We will now consider the closures of the \(d\)-Betti strata.
The stratum of the maximal admissible minimal number of generators \(V_5(L)\) is closed, of codimension \(6\) inside the \(20\)-dimensional \(V(L)\) and irreducible.
The closure of the \(4\)-Betti stratum 
has \(3\) irreducible components, all of them are reduced of codimension \(3\).
The closure \(\overline{V_3(L)}\)
has \(3\) irreducible components, each of codimension \(1\). 
The stratum of complete interesection ideals is dense in \(V(L)\), thus
\(\overline{V_{\textrm{CI}}(L)} = V(L)\).
\end{example}

Next we will display an example where $\Delta(h)>1$:

\begin{example}\label{ex:jump}
 Consider the lex-segment ideal 
\(L = (x^4,x^3y^2, x^2y^3, xy^5, y^7)\) with Hilbert function $h=(1,2,3,4,4,2,1)$ and degree matrix $U$:

\[U =
\left(
\begin{array}{cccc}
2  &   2  &   3   &  4  \\
1  &   1  &   2   &  3  \\
1  &   1  &   2   &  3 \\
0  &   0  &   1   &  2\\
-1  &  -1  &   0   &  1 \\
\end{array} \right)
\]

By \Cref{thm}, $V(L)\simeq\mathbb{A}^{12}$ and any ideal in \(V(L)\) is of the form 
\[
J = I_4
\left(
\begin{array}{cccc}
y^2& 0&          0&      0\\  
-x+c_1y &y &         0&       0\\ 
c_2y & -x & y^2 &     0\\ 
c_3+c_4y & c_5 & -x+c_6y & y^2\\
c_7+c_8y & c_9 & c_{10}+c_{11}y & -x+c_{12}y\\
\end{array} \right),
\]
for some coefficients \(c_1, \dots, c_{12} \in  \res\).

Note that $J$ is homogeneous if and only if, in each entry $(i,j)$ of the matrix above, coefficients of terms with degree other than $u_{i,j}$ vanish, namely $V_{\textrm{hom}}(L)=\mathcal{V}(c_4,c_7,c_8,c_9,c_{11})\simeq \mathbb{A}^7$.

We study the matrix of constant terms 
\begin{equation}\label{eq:Mbar}
 \bar{N} = 
\left(
\begin{array}{cccc}
0 & 0&          0&      0\\  
0 & 0 &         0&       0\\ 
0 & 0 & 0 &     0\\ 
c_3 & c_5 & 0 & 0\\
c_7 & c_9 & c_{10} & 0\\
\end{array} \right)   
\end{equation}

\noindent
 to determine the sets of ideals with a given admissible minimal number of generators:

$$V_5(L)=\mathcal{V}(c_3,c_5,c_7,c_9,c_{10})\simeq \mathbb{A}^7,$$
$$V_4(L)=\mathcal{V}(c_3c_{10},c_5c_{10},c_3c_9-c_5c_7)\backslash V_5(L),$$
$$V_3(L)=\mathbb{A}^{12}\backslash \mathcal{V}(c_3c_{10},c_5c_{10},c_3c_9-c_5c_7).$$
The stratum \(V_5(L)\) with the maximal admissible minimal number of generators is again closed, as in \Cref{ex:3.3}. It is of codimension \(5\) and irreducible.
The closure of \(V_4(L)\) has two irreducible components, both of codimension \(2\). The stratum \(V_3(L)\) with the minimal admissible minimal number of generators is dense in \(V(L)\).

In this example, we can obtain homogeneous ideals in each of the admissible minimal numbers of generators since $\bar{N}$ is projected into $V_{\textrm{hom}}(L)$ just by setting $c_7=c_9=0$, hence its rank can still take any value in $\{0,1,2\}$. For example, take any $J\in V_3(L)$, then its initial ideal

\[
J^\ast = I_4
\left(
\begin{array}{cccc}
y^2& 0&          0&      0\\  
-x+c_1y &y &         0&       0\\ 
c_2y & -x & y^2 &     0\\ 
c_3 & c_5 & -x+c_6y & y^2\\
0 & 0 & c_{10} & -x+c_{12}y\\
\end{array} \right)
\]

\noindent
is minimally generated by a quartic and two quintics and will have the graded resolution $$0\longrightarrow P(-6)\oplus P(-8)\longrightarrow P(-4)\oplus P^2(-5)\longrightarrow P \longrightarrow P/J^\ast \longrightarrow 0,$$

\noindent
which can be obtained from the resolution of $P/L$ by the only two admissible zero cancellations $(P(-6),P(-6))$ and $(P(-7),P(-7))$, in the notation of \cite{RS10}.
\end{example}

\medskip

As an intermediate step towards proving the main result on Betti strata and homogeneous cells, we will reprove some classical results using the parametrization of \Cref{thm}. We will start with the well-known numerical condition on Hilbert functions that admit complete intersections, that is, there exists a complete intersection ideal with that Hilbert function. This is a result from Macaulay in \cite{Mac27}, which has been reproved several times with different techniques in \cite{Bri77,Iar77,RS10}. 

From now on we will assume that $\res$ is a field of characteristic 0.
 This assumption is necessary whenever we want to translate results on $V(\Lex(h))$ to results on \(\Hilb^h(\res[\![x,y]\!])\), or from $V_{\mathrm{hom}}(\Lex(h))$ to results on \(\Hilb^h_{\mathrm{hom}}(\res[\![x,y]\!])\). The computation of dimensions and strata in the Gr\"obner cells themselves hold for any arbitrary field $\res$.

\begin{proposition}\label{prop:CI}[Hilbert functions that admit a complete intersection]
 A Hilbert function $h$ admits a complete intersection if and only if $\Delta(h)=1$. 
\end{proposition}
\begin{proof}
Combining \Cref{lem:rk} and the fact that $\Lex(h)$ is the generic initial ideal, $h$ admits a complete intersection if and only if $\rk\bar{N}$ can reach $t-1$ for some $J=I_t(H+N)$ in $V(\Lex(h))$, where $N \in \T(\Lex(h))$. 
This can only occur if $u_{i,i-2}\leq 0$ for all $i=3,\dots,t+1$.
Since $0<m_1<\dots <m_t$ in the lex-segment case, $u_{i,j}\leq 1$ for any $j\leq i-1$ and it is decreasing in the sense that $u_{i,j}\leq u_{i-1,j},u_{i-1,j+1},u_{i,j+1}$.   
More precisely, $u_{i,j}=1$ iff $u_{i-1,j}=u_{i-1,j+1}=u_{i,j+1}=1$ for $j\leq i-2$. From the fact that $d_k=1$ for $k\geq 2$ if and only if there is a jump in $h$ of height at least 2, we see that $u_{i,i-2}\leq 0$ for all $i=3,\dots,t$ if and only if $\Delta(h)=1$.
\end{proof}

The main point in the proof of \Cref{prop:CI} is to show that for a Hilbert function \(h\) with maximal jump \(\Delta(h)=1\) 
there exists \(N \in \T(\Lex(h))\) with 
\(\rk\bar{N} = t-1\). 
Similarly, for Hilbert functions \(h\) with \(\Delta(h)>1\) we can 
investigate what is the maximal rank of \(\bar{N}\) that can be achieved.

\begin{example}
    Let us consider a Hilbert function $h=(1,2,3,4,5,6,7,8,\dots)$ which starts decreasing at $t=8$ and has maximal jump $\Delta(h)=3$. We want to know which are the possible ranks of a matrix $\bar{N}$ formed by constant terms of $N\in\mathcal{M}(\Lex(h))$. We claim that this only depends on the size of the maximal jump and, more precisely, $0\leq \rk{\bar{N}}\leq 5=8-3$.
    
    To illustrate this, we will consider the following extra assumptions: $h$ has a single jump by 3 (reflected by the fact that there are only two consecutive 1's in the main diagonal of the degree matrix $U$, namely $d_5=d_6=1$) and a single jump by 2 (namely $d_3=1$).
\begin{figure}[h]
\begin{scriptsize}
\[
U=\left(\begin{array}{cccccccc}
    d_1& & & & & & &\\
    1 & d_2 & & & & & &\\
   \ast & 1 & 1 & & & & &\\
   \ast  & 1 & 1 & d_4 & & & &\\
   \ast  &\ast & \ast & \mathbf{1} & \mathbf{1} & \mathbf{1} & & \\
   \ast  &\ast  & \ast  & \mathbf{1} & \mathbf{1} & \mathbf{1} & & \\
    \ast &\ast & \ast & \mathbf{1} & \mathbf{1} & \mathbf{1} & d_7 & \\
   \ast  & \ast & \ast & \ast & \ast  &\ast & 1 & d_8\\
    \ast & \ast & \ast &\ast &\ast  &\ast & \ast & 1\\
    \end{array}\right)
\quad
\bar{N}=
\left(\begin{array}{cccccccc}
    0 & 0 & 0 & 0 & 0 & 0  &  0 & 0\\
    0 & 0 & 0 & 0 & 0 & 0  &  0 & 0\\
    \ast & 0 & 0 & 0 & 0 & 0  &  0 & 0\\
    \ast & 0 & 0 & 0 & 0 & 0 & 0 & 0\\
    \fbox{$\ast$} & \ast & \ast & 0 & 0 & 0 & 0 & 0 \\
    \ast & \fbox{$\ast$} & \ast & 0 & 0 & 0 & 0 & 0 \\
    \ast & \ast  &   \fbox{$\ast$} & \mathbf{0} & 0 & 0 & 0 & 0 \\
    \ast & \ast & \ast &   \fbox{$\ast$} & \ast & \ast & 0 & 0 \\
    \ast & \ast & \ast & \ast &   \fbox{$\ast$} & \ast & \ast & 0 \\
\end{array}\right)
\]
\caption{Degree matrix $U$ and matrix of constant terms $\bar{N}$.}
\label{fig:N}
\end{scriptsize}
\end{figure}

The asteriscs below the diagonal in $U$ represent entries such that $u_{i,j}\leq 0$, hence the corresponding entries in $\bar{N}$ are allowed to be non-zero. Note that only the maximal size of a square of 1's (here 3) determines which is the longest main diagonal of non-zeroes (boxed asteriscs) and hence $\rk(\bar{N})\leq 5$. By considering a matrix $N$ with (possibly zero) constants in the boxed entries and zeroes elsewhere, we can obtain ideals $J=I_8(H+N)$ with Hilbert function $h$ and $4\leq \mu(J)\leq 9$.
\end{example}

This argument is used in the following proposition on the minimal number of generators of an ideal with a given Hilbert function, reproving the results in \cite[Theorem~2.3,\ Theorem~2.4]{Ber09}:

\begin{proposition}[Admissible number of generators of an ideal with a given Hilbert function]
Given $h$, any ideal $J$ such that $\HF_{R/J}=h$ satisfies $\Delta(h)<\mu(J)\leq t+1$, and we can find ideals with any admissible number of generators.
\end{proposition}
\begin{proof}

Combining \Cref{thm} and \Cref{lem:rk}, finding ideals \(J\) with $\Delta(h)<\mu(J)\leq t+1$ is equivalent to proving that 
for \(N \in \T(L)\), with \(L=\Lex(h)\), the rank of \(\bar{N}\) is in the range
$0\leq\rk{\bar{N}}\leq t-\Delta(h)$ and that 
for all the values \(l\) in this range there exists an \(N \in \T(L)\) such 
that \(\rk{\bar{N}}=l\).

In the proof of \Cref{prop:CI} we saw that $u_{i,j}=1$ iff $u_{i-1,j}=u_{i-1,j+1}=u_{i,j+1}=1$ for $j\leq i-2$ and $d_k=1$ for $k\geq 2$ iff $\Delta(h)>1$. Therefore, $\Delta(h)$ is the size of the maximal square of 1's that appears in the degree matrix $U$, see \Cref{fig:N}. It can be checked that it is precisely this size what gives an upper bound the maximal rank of $\bar{N}$: the maximal diagonal that admits non-zero entries has exactly $t-\Delta(h)$ entries (see boxed entries in \Cref{fig:N}). By setting the entries in this diagonal to zero or non-zero values, we obtain matrices with the desired ranks.
\end{proof}
We can also use the parametrization of Gr\"obner cells of lex-segment ideals to recover classical results on the dimension of special subschemes of the punctual Hilbert scheme $\Hilb^n(\res[\![x,y]\!])$. 
The dimension of the subscheme \(\Hilb^h(\res[\![x,y]\!])\) of $\Hilb^n(\res[\![x,y]\!])$ parametrizing ideals with a given Hilbert function was already given in \cite[Theorem III.3.1]{Bri77} and \cite[Theorem~2.12]{Iar77}, where \(\Hilb^h(\res[\![x,y]\!])\) was denoted by \(S(H^{\alpha})\) and \(Z_T\), respectively.

\begin{proposition}[Dimension of the subscheme of the punctual Hilbert scheme with a given Hilbert function] \label{prop:dimCell}
The dimension of the subscheme of \(\Hilb^n(\res[\![x,y]\!])\) corresponding to a certain Hilbert function $h$ is 
\begin{equation}\label{eq:dim}
 \dim (\Hilb^h(\res[\![x,y]\!]))=n-t-\sum_{l\geq 2}n_l\left(\begin{array}{c} l\\2\end{array}\right),   
\end{equation}
\noindent
where $n_l$ denotes the number of jumps of height $l\geq 2$ in $h$. 
\end{proposition}

\begin{proof}
Consider the lex-segment ideal $L$ associated to the Hilbert function $h$. It is immediate from the parametrization of $V(L)$ in \Cref{thm} that $\dim V(L)=\sum_{2\leq j+1\leq i\leq t+1}\left(d_j-\max(u_{i,j},0)\right)$. This expression can be rewritten as $n-t-\#\lbrace u_{i,j}: u_{i,j}=1 \mbox{ for }j\leq i-2 \rbrace$. Using a similar argument as in the proof of \Cref{prop:CI},
1's below the second main diagonal arise only when there is a jump of height $l\geq 2$ in the Hilbert function and they occur in $U$ in disjoint lower triangular blocks of size $l-1$. Therefore,
we derive that 
    the cardinality of the set $\lbrace u_{i,j}: u_{i,j}=1 \mbox{ for }j\leq i-2 \rbrace$ is exactly the sum of the number of jumps of height $l\geq 2$ times $l(l-1)/2$. 
\end{proof}

We reprove the result on the dimension of the subscheme \(\Hilb^h_{\textrm{hom}}(\res[\![x,y]\!])\) that parameterizes homogeneous ideals with a given Hilbert function presented in \cite[Theorem~2.12]{Iar77} and \cite[Corollory~3.1]{CV08}, where \(\Hilb^h_{\textrm{hom}}(\res[\![x,y]\!])\) was denoted by \(G_T\) and \(\mathbb{G}(h)\) respectively:

\begin{proposition}[Dimension of the subscheme of homogeneous ideals with a given Hilbert function]  \label{prop:dimhom}
The dimension of the subscheme of \(\Hilb^h(\res[\![x,y]\!])\) parametrizing homogeneous ideals is 
\begin{equation}\label{eq:dimhom}
\dim(\Hilb^h_{\textrm{hom}}(\res[\![x,y]\!]))=t+\sum_{i=t-1}^s (h_{i-1}-h_i)(h_i-h_{i+1}).
\end{equation}
\end{proposition}
\begin{proof}
By \Cref{thm}, the dimension of the affine space $V_{\hom}(L)$ is the sum of the number of 0's in the degree matrix $U$ and the number of 1's in the strictly lower triangle of $U$ for which their corresponding diagonal entry is strictly higher than 1.
 More precisely, $\dim V_{\hom}(L)=\#\lbrace u_{i,j}: u_{i,j}=0 \mbox{ for }j\leq i-2\rbrace+\#\lbrace u_{i,j}: u_{i,j}=1 \mbox{ for }j\leq i-1 \mbox{ such that }d_j\neq 1\rbrace$.

We start by counting 0's in the degree matrix. For any $3\leq i\leq t+1$, $u_{i,i-2}=0$ if and only if $d_{i-1}=2$. Note that $d_k=2$, for $k\geq 2$, is the indicator that there are two consecutive jumps of height $l_1\geq 1$ and $l_2\geq 1$, respectively, in the Hilbert function (after position $t$). This yields a rectangle of 0's of size $l_1\times l_2$ and hence the number of 0's is the sum of the area of all such rectangles. Therefore,
$$\#\lbrace u_{i,j}: u_{i,j}=0 \mbox{ for }j\leq i-2\rbrace=\sum_{i=t}^s(h_{i-1}-h_i)(h_i-h_{i+1}).$$

Let us now count the number of admissible 1's. If $d_1\neq 1$, there are $t$ such 1's. Indeed, in the case $\Delta(h)=1$, the only 1's are in the third main diagonal and all of them are admissible. Alternatively, if $\Delta(h)>1$, then we will have squares of 1's in the degree matrix whose edge lengths are the size of the jumps, but the admissible ones are only those in the first column of each square, so we still obtain $t$ 1's. 

However, if $d_1=1$, we have to remove the 1's that might be in the first column of $U$. These 1's occur if and only if $h_t<t$ and there will be $t-h_t$ many such 1's. Since $h_{t-2}-h_{t-1}=(t-1)-t=-1$, the number of admissible ones in $U$ is exactly $t+(h_{t-2}-h_{t-1})(h_{t-1}-h_t)=t-(t-h_t)$ for any value of $d_1$. Adding the number of 0's to this quantity completes the proof.
\end{proof}

We are now in the position of stating the parametrization of the Betti strata of $\Hilb^h(\res[\![x,y]\!])$ up to generic change of coordinates. 
Recall that, for the lex-segment ideal \(L\) associated to \(h\), the Gr\"obner cell \(V(L)\) is an affine space of dimension \(D:= n-t-\sum_{l\geq 2}n_l {l\choose 2}\). For the sake of clarity in the notation of the following theorem, $I_d(\bar{N})$ is an ideal in $\res[c_1,..,c_D]$, where $\bar{N}$ is the matrix obtained by taking only the constant terms of a matrix $N\in \T(L)$ in general form. See \eqref{eq:Mbar} for an example of $\bar{N}$.

\begin{theorem}[Betti strata]\label{th:sets_d_gens}
Consider a Hilbert function $h$ and its associated lex-segment ideal $L$. 
For any $\Delta(h) +1\leq d\leq t+1$, the $d$-Betti stratum of the $D$-dimensional affine space $V(L)$, 
is the quasi-affine variety
$$V_d(L)=\mathcal{V}(I_{t+2-d}(\bar{N})) \backslash \mathcal{V}(I_{t+1-d}(\bar{N}))\subset V(L)\simeq \mathbb{A}^D.$$
\noindent
In particular, $V_{t+1}(L)$ is an affine space and 
$$V_{\Delta(h)+1}(L)=\mathbb{A}^D\backslash
\mathcal{V}(I_{t-\Delta(h)}(\bar{N}))$$
\noindent
is a \(D\)-dimensional quasi-affine variety.
\end{theorem}

Note that a consequence of this result is that ideals that have the minimum minimal number of generators admissible for a given $h$, namely those in $V_{\Delta(h)+1}(L)$, are dense in $\Hilb^h(\res[\![x,y]\!])$, as proved in \cite[Proposition III.2.1]{Bri77}.
In particular, if $\Delta(h)=1$, then $V_{\textrm{CI}}(L)$ is the quasi-affine variety obtained by removing $t-1$ hyperplanes from $\mathbb{A}^D$.

\begin{proposition}[Homogeneous sub-cell]\label{prop:hom}
$V_{\mathrm{hom}}(L)$ is a $D_h$-dimensional affine space, where $D_h$ is the right-hand side of \eqref{eq:dimhom}. Moreover, for any $J=I_t(H+N)$ its initial ideal $J^\ast$ can be computed by projecting $N$ into that affine space. 
\end{proposition}
\begin{proof}
The maximal minors $f_0,\dots,f_t$ of $H+N$ are a $\otau$-enhanced standard basis of $J=I_t(H+N)$ by construction. By \Cref{rem:localstructure}, they also form a standard basis of $J$, namely $J^\ast=(f_0^\ast,\dots,f_t^\ast)$. 
Let us call $N^\ast$ the matrix obtained by only keeping the terms of degree $u_{i,j}$ in each entry $N_{i,j}$.
Note that $N^\ast\in \T(L)$ and $I_t(H+N^\ast)=J^\ast$. 

We claim that an ideal $J=I_t(H+N)$, for some $N\in \T(L)$, is homogeneous if and only the maximal minors $f_0,\dots,f_t$ of $H+N$ are homogeneous.
Indeed, homogeneity of a system of generators yields homogeneity of $J$. Conversely, if $J$ is homogeneous then $J=I_t(H+N)=I_t(H+N^\ast)$, thus $N=N^\ast$ by \Cref{thm}.
\end{proof}

\section{A Cellular Decomposition of the punctual Hilbert scheme}\label{sec:conjecture}

In this section we leave the class of lex-segment ideals and consider general zero-dimensional 
monomial ideals \(E\) in \(\res[\![x,y]\!]\). 
The collection of all Gr\"obner cells \(V(E)\) forms a cellular decomposition of the
punctual Hilbert scheme \(\Hilb^n(\res[\![x,y]\!])\).
A cellular decomposition which additionally respects the Hilbert function stratification induces a cellular decomposition 
of \(\Hilb^h(\res[\![x,y]\!])\), the stratum of \(\Hilb^n(\res[\![x,y]\!])\) with a prescribed Hilbert function \(h\). 
When working over \(\mathbb{C}\) a cellular decomposition of a scheme can be used to calculate its Betti numbers. This was done in \cite{ES} 
for the Hilbert schemes of points of \(\mathbb{P}^2\), \(\mathbb{A}^2\) and the punctual Hilbert scheme \(\Hilb^n(\res[\![x,y]\!])\). In \cite{Got90} this result is refined for the stratum \(\Hilb^h(\res[\![x,y]\!])\). Both papers
use representation theoretical
methods without giving explicit parametrizations of the cells.
We conjecture what should be the parametrization of Gr\"obner cells for general \(E\) in terms of Hilbert-Burch matrices and give some detailed examples and evidence for our conjecture.

In \cite[Proposition~5.1/Definition~5.2]{HW21} we showed that the set of matrices \(\T(E)\) from \Cref{def:setT2} parametrizes a subset of \(V(E)\),
namely those ideals \(J\) in \(V(E)\) that admit a \(\otau\)-enhanced standard basis that also forms a Gr\"obner basis of the ideal \(J \cap \res[x,y]\) with respect to the lexicographic ordering. We proved that for the class of monomial ideals \(E=(x^t,x^{t-1}y^{m_1},\dots, y^{m_t})\) that satisfy
\begin{equation}\label{lexGBcond}
m_j-j-1 \le m_i-i    \mbox{ for all } j<i,
\end{equation} this is already the whole Gr\"obner cell \(V(E)\). 
The class includes lex-segment ideals, but for example also allows \(m_i=m_{i+1}\) for exactly one \(i\).
For general \(E\) there exist ideals \(J \in V(E)\) which do not admit such a \(\otau\)-enhanced standard basis. 
What occurs in this situation is that \(\Ltlex(J \cap \res[x,y]) \not= E\), for an example check \cite[Lemma~5.12,~Example~5.13]{HW21}. 
In these cases, to obtain \(J\) as the ideal of maximal minors of \(H+N\) it is no longer enough to consider lower triangular matrices \(N\).

\begin{definition}\label{def:setTgen}
For a monomial ideal \(E \subset \res[\![x,y]\!]\) we denote by $\Nd(E)$ the set of matrices $N=(n_{i,j})$ of size $(t+1)\times t$ with entries in $\res[y]$ such that its non-zero entries satisfy
\begin{itemize}
\item $u_{i,j}+1 \leq \ord(n_{i,j}) \leq \deg(n_{i,j}) <d_i$ for any $i\leq j$, 
\item $u_{i,j} \leq \ord(n_{i,j}) \leq\deg(n_{i,j})<d_j$ for any $i> j$.
\end{itemize}
\end{definition}
Note that the set \(\T(E)\) is a  subset of \(\Nd(E)\), 
and for ideals satisfying condition \eqref{lexGBcond} \(\T(E)=\Nd(E)\), since the inequality \(u_{i,j}+1>d_i\) holds for \(i \le j\), and thence yields \(n_{i,j}=0\) for all \(i \leq j\).

\begin{conjecture}\cite[Conjecture~5.14]{HW21}
\label{Conj}
Let $E \subset \res[\![x,y]\!]$ be a zero-dimensional monomial ideal. 
Then
$$\begin{array}{r r c l}
\Phi_E: & \Nd(E) & \longrightarrow & V(E)\\
& N & \longmapsto & I_t(H+N)
\end{array}$$

\noindent 
is a bijection. 
\end{conjecture}

By \cite[Lemma~4.4]{HW21} it is clear that \(\Phi_E\) is well-defined and that the minors form a \(\otau\)-enhanced standard basis of \(I_t(H+N)\).
By \Cref{thm}, \Cref{Conj} holds for lex-segment ideals.
If the conjecture holds for all monomial ideals, we obtain a cellular decomposition of the whole punctual Hilbert scheme 
and can define a canonical Hilbert-Burch matrix for all ideals without considering a change of coordinates.
This decomposition induces a cellular decomposition of \(\Hilb^h(\res[\![x,y]\!])\). 
This is not the case for the cellular decomposition given in \cite{CV08} 
because the Hilbert function of ideals in a cell can vary, see \cite[Example~5.11]{HW21}.

Note that even without a proof of \Cref{Conj}, we can still prove a version of \Cref{prop:hom} for homogeneous sub-cells of Gr\"obner cells of general \(E\). This result agrees with \cite[Corollary~1]{CV08} and shows that even though their Gr\"obner cells differ, the subset of homogeneous ideals agrees. 
This is no surprise, since for homogeneous \(f\) we have that \(\Lt(f)=\Ltlex(f)\) by definition.

\begin{lemma} \label{homo}
The homogeneous sub-cell \(V_{\mathrm{hom}}(E)\) of \(V(E)\) is an affine space of dimension \(\# \{(i,j) \ | \ i >j, 0\le u_{i,j} < d_j \}\).
\end{lemma}
\begin{proof}
Homogeneous ideals \(J\) in \(V(E)\) are inside the set of ideals in \(V(E)\) 
admitting a canonical Hilbert-Burch matrix. This is clear since for homogeneous \(f\), 
the leading terms \(\Lt(f) = \Ltlex(f)\) coincide, so a homogeneous \(\otau\)-enhanced standard basis of \(J\)
will also form a lex-Gr\"obner basis of \(J \cap \res[x,y]\).
Thence by \cite[Proposition~5.1]{HW21} 
\(V_{\hom}(E) \subset \phi_E(\T(E))\), and there exists a unique \(N\) such that \(J=I_t(H+N)\). 
Now \(J\) is homogeneous if and only if  \(N \in \T(E)\) has only terms of degree \(u_{i,j}\). 
Thus \(V_{\hom}(E)\) is an affine space of dimension \( \# \{(i,j) \ | \ i >j, 0 \le u_{i,j} < d_j \}\).
\end{proof}

\begin{example}[A cellular decomposition of {$\Hilb^6(\res[\![x,y]\!])$}] \label{ex:n=6}
Let us consider \(n=6\). 
There are eleven partitions of 6, so the punctual Hilbert scheme \(\Hilb^6(\res[\![x,y]\!])\) has a cellular decomposition into eleven cells. 
There are four possible Hilbert functions \([1,2,3]\), \([1,2,2,1]\), \([1,2,1,1,1]\) and \([1,1,1,1,1,1]\). 

Note that this is significantly different from Briançon's table in \cite{Bri77} because there the author provides a representative of all possible analytic types of ideals in $\Hilb^6(\res[\![x,y]\!])$, whereas our cells contain ideals with a common leading term ideal but in general different analytic types coexist in the same cell (for example, ideals with different minimal number of generators).

The Hilbert function \([1,2,3]\) is only attained by the ideal \((x^3,x^2y,xy^2,y^3) = (x,y)^3\). 
By Proposition \ref{prop:dimCell} the dimension of this stratum is \[n-t-\sum_{l\geq 2}n_l\left(\begin{array}{c} l\\2\end{array}\right) = 6 - 3 - 1 \cdot \left(\begin{array}{c} 3\\2\end{array}\right)=0,\] and the 
Gr\"obner cell consists only of the single point corresponding to the monomial ideal itself. It is minimally generated by four elements.

There are six cells with Hilbert function \([1,2,2,1]\), see \Cref{1221}. 
Admissible parameters to obtain a homogeneous ideal, i.e. those of \(V_{\hom}(E)\), are indicated in bold.
The Hilbert function has only jumps of height 1, so the maximal dimensional cell among them, corresponding to \(m=(2,4)\), has dimension \(n-t = 6-2 = 4\). 
General ideals in this cell will be minimally generated by two elements. If \(c_2 =0\) the ideal is generated by three elements.

The other five cells with this Hilbert function are not lex-segment. The cells \([2,2,2]\) and \([3,3]\), corresponding 
to monomial complete intersection ideals \((x^3,y^2)\) and \((x^2,y^3)\), consist only of complete intersection ideals. 
The two cells \([1,1,2,2]\) and \([1,1,1,3]\) only contain ideals that are minimally generated by 3 elements while 
the cell \([1,1,4]\) contains ideals \(I,J\) with \(\mu(I) = 2 \) and \(\mu(J)=3\). Notice that \(\dim(V(E)) - \dim(V_{\hom}(E)) = 1\) for all \(E\) with Hilbert function
\([1,2,2,1]\).

\begin{figure}
\[
\begin{tabular}{c c c c c} \toprule m & boxes  & H+N & dimension & \(\mu\)  \\
  \toprule
    \hline
  & & [1,2,3] \\
  \hline
  \\
   \([1, 2, 3]\) &
  \( \begin{tikzpicture}[scale=0.25]
  \draw (0,0) -- (3,0) -- (3,1) -- (2,1) -- (2,2) -- (1,2) -- (1,3) -- (0,3) -- cycle;
  \draw (0,1) -- (2,1);
  \draw (0,2) -- (1,2);
  \draw (1,0) -- (1,3);
  \draw (2,0) -- (2,2);
  \end{tikzpicture}\)
  & \(
\left(
\begin{array}{c c c}
  y & 0 & 0\\
  -x & y & 0\\
  0 & -x & y\\
  0 & 0 & -x
\end{array}
\right)
\) & 0 & 4\\
\\

\hline
 & & [1,2,2,1] \\
 \hline
 \\
  \([1, 1, 2, 2]\) &
  \( \begin{tikzpicture}[scale=0.25]
  \draw (0,0) -- (0,2);
  \draw (1,0) -- (1,2);
  \draw (2,0) -- (2,2);
  \draw (3,0) -- (3,1);
  \draw (4,0) -- (4,1);
  \draw (0,2) -- (2,2);
  \draw (0,1) -- (4,1);
  \draw (0,0) -- (4,0);
  \end{tikzpicture}\)
 & \(
\left(
\begin{array}{c c c c}
  y & 0 & 0 & c_1\\
  -x & 1 & 0 & 0\\
  0 & -x & y & 0\\
  0 & 0 & -x & 1\\
  0 & 0 & 0 & -x
\end{array}
\right)
\) & 1 & 3\\
\\
  \([1, 1, 1, 3]\) &
  \( \begin{tikzpicture}[scale=0.25]
  \draw (0,0) -- (0,3);
  \draw (1,0) -- (1,3);
  \draw (2,0) -- (2,1);
  \draw (3,0) -- (3,1);
  \draw (4,0) -- (4,1);
  \draw (0,0) -- (4,0);
  \draw (0,1) -- (4,1);
  \draw (0,2) -- (1,2);
  \draw (0,3) -- (1,3);
  \end{tikzpicture}\)
   & \(
\left(
\begin{array}{c c c c}
  y & 0 & c_1 & 0\\
  -x & 1 & 0 & 0\\
  0 & -x & 1 & 0\\
  0 & 0 & -x & y^2\\
  0 & 0 & 0 & -x + \bm{c_2} y 
\end{array}
\right)
\) & 2 & 3\\
\\
  \([2, 2, 2]\) &
    \( \begin{tikzpicture}[scale=0.25]
  \draw (0,0) -- (0,2);
  \draw (1,0) -- (1,2);
  \draw (2,0) -- (2,2);
  \draw (3,0) -- (3,2);
  \draw (0,0) -- (3,0);
  \draw (0,1) -- (3,1);
  \draw (0,2) -- (3,2);
  \end{tikzpicture}\)

   & \(
\left(
\begin{array}{c c c}
  y^2 & 0 &  c_1 y\\
  -x + \bm{c_2} y  & 1 & 0\\
  0 & -x & 1\\
  0 & 0 & -x
\end{array}
\right)
\) & 2 & 2\\
\\
  \([1, 1, 4]\)&
  \( \begin{tikzpicture}[scale=0.25]
  \draw (0,0) -- (0,4);
  \draw (1,0) -- (1,4);
  \draw (2,0) -- (2,1);
  \draw (3,0) -- (3,1);
  \draw (0,4) -- (1,4);
  \draw (0,0) -- (3,0);
  \draw (0,1) -- (3,1);
  \draw (0,2) -- (1,2);
  \draw (0,3) -- (1,3);
  \end{tikzpicture}\)

   & \(
\left(
\begin{array}{c c c}
  y & 0 & 0\\
  -x & 1 & 0\\
  0 & -x & y^3\\
  \bm{c_1} & 0 & -x + \bm{c_2} y + c_3 y^2  
\end{array}
\right)
\) & 3 & 2,3\\
\\
  \([3, 3]\) &
      \( \begin{tikzpicture}[scale=0.25, rotate=90]
  \draw (0,0) -- (0,2);
  \draw (1,0) -- (1,2);
  \draw (2,0) -- (2,2);
  \draw (3,0) -- (3,2);
  \draw (0,0) -- (3,0);
  \draw (0,1) -- (3,1);
  \draw (0,2) -- (3,2);
  \end{tikzpicture}\)

  & \(
\left(
\begin{array}{c c}
  y^3 & 0\\
  -x + \bm{c_1} y + c_2 y^2  & 1\\
  \bm{c_3} y^2  & -x
\end{array}
\right)
\) & 3 & 2 \\
\\
  \([2, 4]\) &
    \( \begin{tikzpicture}[scale=0.25]
  \draw (0,0) -- (2,0);
  \draw (0,1) -- (2,1);
  \draw (0,2) -- (2,2);
  \draw (0,3) -- (1,3);
  \draw (0,4) -- (1,4);
  \draw (2,0) -- (2,2);
  \draw (1,0) -- (1,4);
  \draw (0,0) -- (0,4);
  \end{tikzpicture}\)

  & \(
\left(
\begin{array}{c c}
  y^2 & 0\\
  -x + \bm{c_1} y  & y^2\\
   \bm{c_2} + c_3 y & -x +\bm{c_4} y 
\end{array}
\right)
\) & 4 & 2,3 \\
\end{tabular}
\]
\caption{Gr\"obner cells for Hilbert functions [1,2,3] and [1,2,2,1]} \label{1221}
\end{figure}

There are two cells with Hilbert function \([1,2,1,1,1]\), see \Cref{12111}. Both contain ideals minimally generated by three elements 
(in case \(c_3 =0\) for \(E=(x^5,xy,y^2)\) or \(c_1=0\) for \(E=(x^2,xy,y^5)\)) and complete intersection ideals.
\begin{figure}
\[
\begin{tabular}{c c c c c} \toprule m & boxes  & H+N & dimension & \(\mu\)  \\
  \toprule
  \hline
\hline
& & \([1, 2, 1, 1, 1]\)\\
\hline
\\
\([1, 1, 1, 1, 2]\)&

 \( \begin{tikzpicture}[scale=0.25]
  \draw (0,0) -- (0,2);
  \draw (1,0) -- (1,2);
  \draw (2,0) -- (2,1);
  \draw (3,0) -- (3,1);
  \draw (4,0) -- (4,1);
  \draw (5,0) -- (5,1);
  \draw (0,0) -- (5,0);
  \draw (0,1) -- (5,1);
  \draw (0,2) -- (1,2);
  \end{tikzpicture}\)
   & \(
\left(
\begin{array}{c c c c c}
  y & 0 & c_1 & c_2 & c_3\\
  -x & 1 & 0 & 0 & 0\\
  0 & -x & 1 & 0 & 0\\
  0 & 0 & -x & 1 & 0\\
  0 & 0 & 0 & -x & y\\
  0 & 0 & 0 & 0 & -x
\end{array}
\right)
\) & 3 & 2,3 \\
\\
  \([1, 5]\) &
 \( \begin{tikzpicture}[scale=0.25]
  \draw (0,0) -- (2,0);
  \draw (0,1) -- (2,1);
  \draw (0,2) -- (1,2);
  \draw (0,3) -- (1,3);
  \draw (0,4) -- (1,4);
  \draw (0,5) -- (1,5);
  \draw (0,0) -- (0,5);
  \draw (1,0) -- (1,5);
  \draw (2,0) -- (2,1);
  \end{tikzpicture}\)

  & \(
\left(
\begin{array}{c c}
  y & 0\\
  -x & y^4\\
  c_1 & -x + \bm{c_2} y + c_3 y^2 + c_4 y^3 
\end{array}
\right)
\) & 4 &2,3\\
\\
\end{tabular}\]
\caption{Gr\"obner cells with Hilbert function [1,2,1,1,1]}\label{12111}
\end{figure}

The Hilbert function \([1,1,1,1,1,1]\) is attained by the monomial ideals \((x^6,y)\) and the lex-segment monomial ideal 
\((x,y^6)\). Both cells consist only of complete intersection ideals and have dimensions \(4\) and \(5\), see \Cref{111111}.
The cell of \(m=[6]\) is the dense open subset of the punctual Hilbert scheme \(\Hilb^6(\res[\![x,y]\!])\). 

Note that here the difference between the dimension of \(V(E)\) and \(V_{\hom}(E)\) is again constant for all \(E\) with the same Hilbert function: \(\dim(V(E))-\dim(V_{\hom}(E))\) is \(3\) for \([1,2,1,1,1]\) and \(4\) for \([1,1,1,1,1,1]\).

\begin{figure}
\[
\begin{tabular}{c c c c c} \toprule m & boxes  & H+N & dimension & \(\mu\)  \\
  \toprule
\hline
& & \([1, 1, 1, 1, 1, 1]\)\\
\hline \\
\([1]^6\) &

 \( \begin{tikzpicture}[scale=0.25]
  \draw (0,0) -- (0,1);
  \draw (1,0) -- (1,1);
  \draw (2,0) -- (2,1);
  \draw (3,0) -- (3,1);
  \draw (4,0) -- (4,1);
  \draw (5,0) -- (5,1);
  \draw (6,0) -- (6,1);
  \draw (0,0) -- (6,0);
  \draw (0,1) -- (6,1);
  \end{tikzpicture}\)

& \(
\left(
\begin{array}{c c c c c c}
  y & 0 & c_1 & c_2 & c_3 & c_4\\
  -x & 1 & 0 & 0 & 0 & 0\\
  0 & -x & 1 & 0 & 0 & 0\\
  0 & 0 & -x & 1 & 0 & 0\\
  0 & 0 & 0 & -x & 1 & 0\\
  0 & 0 & 0 & 0 & -x & 1\\
  0 & 0 & 0 & 0 & 0 & -x
\end{array}
\right)
\) & 4 & 2\\
  \([6]\) &
 \( \begin{tikzpicture}[scale=0.25, rotate=90]
  \draw (0,0) -- (0,1);
  \draw (1,0) -- (1,1);
  \draw (2,0) -- (2,1);
  \draw (3,0) -- (3,1);
  \draw (4,0) -- (4,1);
  \draw (5,0) -- (5,1);
  \draw (6,0) -- (6,1);
  \draw (0,0) -- (6,0);
  \draw (0,1) -- (6,1);
  \end{tikzpicture}\)
   & \(
\left(
\begin{array}{c}
  y^6\\
  -x + \bm{c_1} y + c_2 y^2 + c_3 y^3 + c_4 y^4 + c_5 y^5 
\end{array}
\right)
\) & 5 &2 \\
\\
\hline
\end{tabular}
\]
    \caption{Gr\"obner cells with Hilbert function [1,1,1,1,1,1]}
    \label{111111}
\end{figure}

As stated before, for \(E\) not satisfying condition \eqref{lexGBcond} the set \(\Nd(E)\) contains matrices with non-zero entries above the diagonal. This is the case for \(m = (1,1,2,2), (1,1,1,3), (2,2,2), (1,1,1,1,2)\) and \((1,1,1,1,1,1)\). 
For those ideals we checked by comparing to the reduced standard basis that \Cref{Conj} gives a parametrization of the Gr\"obner cell.

We can investigate the dimensions of the occurring cells. When we define \(a_i\) as the number of cells of dimension \(i\), we obtain \[a = (a_0,a_1,a_2,a_3,a_4,a_5) = (1,1,2,3,3,1).\]
This vector is an invariant of the space. 
\end{example}

When a scheme \(X\) over \(\mathbb{C}\) has a cellular decomposition with dimension vector \(a\) as defined in \Cref{ex:n=6}, then all other cellular decompositions of \(X\) will have the same dimension vector. 
It holds that \(a_i=b_{2i}(X)\), the \(2i\)-th Betti number of \(X\), and that the Betti numbers \(b_{2i+1}(X)=0\), see \cite[Theorem~4.4/4.5]{Bia73} and \cite[Chapter~19.1]{FU98}.

In \cite{ES} this method is used to calculate the Betti numbers of the Hilbert scheme of points of the projective plane, the affine plane and, of most interest to us, the punctual Hilbert scheme \(\Hilb^n(\res[\![x,y]\!])\). 
To state their theorem, we need the following definition.

\begin{definition}
Let \(l,n \in \mathbb{Z}_{>0}\), then we define \(P(n,l)\) as the number of partitions of \(n\) bounded by \(l\), i.e. as the number of sequences \(0 =m_0 \le m_1 \le \dots \le m_t \le l\) such that \(\sum_{i=1}^t m_i = n\). 
\end{definition}
\begin{theorem}\cite[Theorem~1.1~(iv)]{ES} The non-zero Betti numbers of the punctual Hilbert scheme are
\[b_{2i}(\Hilb^n(\res[[x,y]]))=P(i,n-i).\]
\end{theorem}

This result gives us a way of checking \Cref{Conj} for plausibility by checking if the number of 
monomial ideals \(E\) with \(\dim(\Nd(E))=i\) is equal to \(P(i,n-i)\).

We created a Julia (\cite{julia}) module that uses OSCAR.jl (\cite{OSCAR}) to calculate examples and perform this plausibility check. 
The module and Jupyter-notebooks with experiments are available at \url{https://github.com/anelanna/LocalHilbertBurch.jl}. 

The module can calculate the (conjectural) cellular decomposition for a given \(n\). For all partitions of \(n\) it creates a \texttt{Cell} that has the following properties: 
\begin{itemize}
    \item \texttt{m} -- the partition,  
    \item \texttt{E} -- its associated monomial ideal, 
    \item \texttt{d} -- the vector of differences, 
    \item \texttt{U} -- the degree matrix, 
    \item \texttt{hilb} -- the Hilbert function of the ideals in the cell, 
    \item \texttt{H} -- the canonical Hilbert-Burch matrix of the monomial ideal,
    \item \texttt{M} -- the general canonical Hilbert-Burch matrix \(H+N\),
    \item \texttt{N} -- the corresponding element in \(\Nd(E)\), 
    \item \texttt{I} -- the maximal minors of \(M=H+N\), and
    \item \texttt{dim} -- the dimension of the cell.
\end{itemize}

The vector of dimensions \(a\) can be computed with \texttt{sorted\_celllist(n)}, a function that returns a dictionary mapping an integer \(i\) to a vector with the cells of dimension \(i\).

The numbers of bounded partitions can be found in the On-Line Encyclopedia of Integer Sequences, \cite{OEIS},
as (diagonals in) A008284, or relabelled, and thus serving our purposes a bit better, as (diagonals in) A058398 (see \url{https://oeis.org/A008284} and \url{https://oeis.org/A058398)}. 

For \(n \le 50\) we calculated all dimension vectors of our proposed cellular decomposition and checked that they are the correct ones. 
A stronger evidence for \Cref{Conj} would be that \(\phi_E: \Nd(E) \to V(E)\) is an isomorphism of affine
spaces for all \(n \le 40\). We have checked this for \(n \le 7\) where we described \(V(E)\) by means
of the reduced \(\otau\)-enhanced standard bases.

\begin{example}[\Cref{ex:n=6} continued]
We can use the connection between cellular decompositions and Betti numbers described above
to calculate the Betti numbers of \(\Hilb^{h}(\res[\![x,y]\!])\) for \(\res=\mathbb{C}\) and \(h\) an admissible Hilbert function in the \(n=6\) case. We obtain the following Betti numbers:
\[
b_i(\Hilb^{(1,2,3)}(\res[\![x,y]\!])) = \begin{cases}
1, & i = 0;\\
0, & \mbox{otherwise}.
\end{cases}
\]

\[
b_i(\Hilb^{(1,2,2,1)}(\res[\![x,y]\!])) = \begin{cases}
 1, & i =2,8;\\
 2,& i =4,6;\\
0, & \mbox{otherwise}.
\end{cases}
\]

\[
b_i(\Hilb^{(1,1,1,2,1)}(\res[\![x,y]\!])) = \begin{cases}
 1, & i =6,8;\\
0, & \mbox{otherwise}.
\end{cases}
\]

\[
b_i(\Hilb^{(1,1,1,1,1,1)}(\res[\![x,y]\!])) = \begin{cases}
1, & i = 8,10;\\
0, & \mbox{otherwise.}
\end{cases}
\]
\end{example}

\begin{remark}
    More generally, we can describe the subscheme \(\Hilb^{(1,1,\dots,1)}(\res[\![x,y]\!])\)
    of the punctual Hilbert scheme for any \(n\). 
    There are only two cells with this Hilbert function, namely \(L=(x,y^n)\) and \(E=(x^n,y)\). 
    The dimension of the cell of the lex-segment ideal \(L\) is \(n-1\). 
    By considering the reduced \(\otau\)-enhanced standard basis of \(E\) we can verify that 
    \Cref{Conj} holds for \(E\) or just directly show that \(\dim(V(E))=n-2\).
    So the cellular decomposition has one cell of dimension \(n-2\) and one of dimension \(n-1\).
    Hence, the only non-vanishing Betti numbers are 
    \(b_{2(n-2)}(\Hilb^{(1,\dots,1)}(\res[\![x,y]\!]))=b_{2(n-1)}(\Hilb^{(1,\dots,1)}(\res[\![x,y]\!]))=1\).
    This agrees with the observations in \cite[Remark~2.2a]{Got90}.
\end{remark}

Another relevant observation from \Cref{ex:n=6} is that within each stratum
\(\Hilb^h(\res[\![x,y]\!])\) of the punctual Hilbert scheme 
\(\Hilb^n(\res[\![x,y]\!])\), the difference of dimensions between any local Gr\"obner cell and its homogeneous sub-cell is constant.
This fact actually follows from a result by Iarrobino:

\begin{theorem}{\cite[Theorem~2.11,~Theorem~3.14]{Iar77}}
The stratrum \(\Hilb^h(\res[\![x,y]\!])\) is a locally trivial bundle over
\( \Hilb^h_{\hom}(\res[\![x,y]\!])\)  having fibre, an affine space,
and a global section.
\end{theorem}

The cellular decomposition of \(\Hilb^h(\res[\![x,y]\!])\) given by all Gr\"obner cells \(V(E)\) where \(E\) is a monomial ideal with Hilbert
function \(h\), induces a cellular decomposition of
\(\Hilb^h_{\hom}(\res[\![x,y]\!])\) by taking homogeneous sub-cells \(V_{\hom}(E)\).  By a combination of \Cref{prop:hom} and 
\Cref{homo} the fibration restricts to  \(V(E) \to V_{\hom}(E)\).
Since we know the dimension of \(V_{\hom}(E)\) by \Cref{homo}, we can
check (at least in examples) whether the difference
\(\dim(\Nd(E))-\dim(V_{\hom}(E))\) is constant for all \(E\) with the same Hilbert
function. If this holds, it implies that our conjectured cells have the right dimension.
Using our \texttt{Julia} module, we have checked this for all strata
\(\Hilb^h(\res[\![x,y]\!])\) of \(\Hilb^n(\res[\![x,y]\!])\) with \(n \le 50\), thus providing strong evidence for \Cref{Conj}.

\section*{Acknowledgements} 
We want to thank Lars Kastner for helping us to set up the Julia module and helping with technical issues. 
We thank the two anonymous referees for many useful comments on a preliminary version of this paper.
The first author wants to thank Anthony Iarrobino, Pedro Macias Marques, Maria Evelina Rossi, and Jean Vallès, organizers of the AMS-EMS-SMF Special Session on Deformation of Artinian Algebras and Jordan Types, for their invitation to the session.
\bibliographystyle{amsplain}
\bibliography{references}
\end{document}